\documentclass[12pt]{amsart}
\usepackage{amssymb}
\usepackage{hyperref}

\newtheorem{thm}{Theorem}[section]
\newtheorem{lem}[thm]{Lemma}

\newtheorem{prop}[thm]{Proposition}
\newtheorem{ex}[thm]{Example}

\newtheorem*{prob*}{Open problem}

\theoremstyle{definition}
\newtheorem{defi}[thm]{Definition}

\theoremstyle{remark}
\newtheorem{rem}[thm]{Remark}
\newtheorem*{rem*}{Remark}

\newcommand{\kringel}{\mathbin{\raise1pt\hbox{$\scriptstyle\circ$}}}
\newcommand{\pkt}{\mathbin{\raise0pt\hbox{$\scriptstyle\bullet$}}}

\newcommand{\CB}{\mathcal{B}}

\newcommand{\CG}{\mathcal{G}}

\newcommand{\al}{\alpha}
\newcommand{\be}{\beta}
\newcommand{\ga}{\gamma}
\newcommand{\de}{\delta}

\renewcommand{\phi}{\varphi}

\begin{document}

% Ab hier duerfen Sie wieder.

\title[Jacobi-Jordan algebras]{Jacobi-Jordan algebras}
%  Die Kurzfassung kommt oben ueber die Seiten, sie steht in eckigen Klammern
%  Auch Autorennamen koennen eine Kurzfassung haben

\author[D. Burde]{Dietrich Burde}
\author[A. Fialowski]{Alice Fialowski}
\address{Fakult\"at f\"ur Mathematik\\
Universit\"at Wien\\
Oskar-Morgenstern-Platz 1\\
1090 Wien \\
Austria}
\email{dietrich.burde@univie.ac.at}
\address{E\"otv\"os Lor\'and University\\
Institute of Mathematics\\
P\'azm\'any P\'eter s\'et\'any 1/C\\
1117 Budapest, Hungary}
\email{fialowsk@cs.elte.hu }

\date{\today}

%\thanks{The first author was supported by the FWF, Projekt P21683.}
%\thanks{The second author expresses his gratitude towards ...}

\begin{abstract}
We study finite-dimensional commutative algebras, which satisfy 
the Jacobi identity. Such algebras are Jordan algebras. We describe some
of their properties and give a classification in dimensions $n<7$ 
over algebraically closed fields of characteristic not $2$ or $3$.
\end{abstract}

\maketitle

\section{Introduction}

Finite-dimensional commutative associative algebras (with and without $1$) 
have been studied intensively. In particular, many different classification results
have appeared in the literature, see \cite{POO}, \cite{DGR}, \cite{FIP1}, \cite{FIP2} for recent
publications, and the references given therein. On the other hand, many important
classes of commutative {\it non-associative} algebras have been studied, too. 
Examples are Jordan algebras, commutative power-associative algebras, commutative
alternative algebras, Bernstein algebras and many more. We mention also
pre-Lie algebras, Novikov algebras and LR-algebras, which we studied in a geometric context, 
see \cite{BU22}, \cite{BU24}, \cite{BU34}. In each case there, the subclass of
commutative algebras therein automatically satisfies associativity. In particular, the 
classification of such algebras includes the classification of commutative associative algebras
as a special case. See for example the classification of Novikov algebras in low 
dimension \cite{BU42}. \\
Here we want to study another class of commutative non-associative algebras. They satisfy
the Jacobi identity instead of associativity. It turns out that they are a special class 
of Jordan nilalgebras. For this reason we want to call them {\it Jacobi-Jordan algebras}. 
These algebras are also related to Bernstein-Jordan algebras.
Although they seem to be similar to Lie algebras at first sight, these algebras are quite different. \\
We are grateful to Manfred Hartl, who asked what is known about commutative algebras sa\-tisfying the
Jacobi identity. We are also grateful to Ivan Shestakov for answering
some questions on the subject, and paying our attention to
Bernstein-Jacobi algebras.
% who pointed out to us the unpublished work of Sergei Sverchkov 
%on a longstanding problem, whether such algebras are always isomorphis to a subalgebra of 
%$A^+$, the symmetric Jordan algebra of an associative algebra $A$. Sverchkov gave a negative answer.

\section{Jacobi-Jordan algebras}

Let $A$ be a finite-dimensional algebra over a field $K$. We will always assume that the  
characteristic of $K$ is different from $2$ or $3$, although some of our results also hold
for characteristic $2$ or $3$.  
We denote the algebra product by $x\cdot y$. The principal powers of an element $x\in A$
are defined recursively by $x^1=x$ and $x^{i+1}=x\cdot x^i$. Also, let $A^1=A$ and define
$A^k=\sum_{i+j=k}A^iA^j$ recursively for all integers $k\ge 2$. We say that $A$ is {\it nilpotent},
if there is a positive integer $k$ with $A^k=(0)$. The algebra $A$ is called a {\it nilalgebra} of nilindex $n$, 
if the subalgebra generated by any given $a\in A$ is nilpotent, i.e., if $y^n=0$ for all $y\in A$, and there 
exists an element $x\in A$ with $x^{n-1}\neq 0$. If an algebra $A$ is a commutative nilalgebra of nilindex $n\le 2$, 
then all products vanish because of $x\cdot y=\frac{1}{2}((x+y)^2-x^2-y^2))=0$ for all $x,y\in A$.

\begin{defi}\label{jj}
An algebra $A$ over a field $K$ is called a {\it Jacobi-Jordan} algebra if it
satisfies the following two identities,
\begin{align}
x\cdot y -y\cdot x & = 0 \label{1}\\
x\cdot (y\cdot z)+y\cdot (z\cdot x)+z\cdot (x\cdot y) & = 0 \label{2} 
\end{align}
for all $x,y,z\in A$.
\end{defi}

In other words, Jacobi-Jordan algebras are commutative algebras which satisfy the
Jacobi identity. In the case of Lie algebras, the product is not commutative but anti-commutative.
Then the Jacobi identity has equivalent formulations which are not necessarily equivalent in
the commutative case. Obviously the Jacobi identity implies that $x^3=0$ for all $x\in A$, since $3\neq 0$. 
Hence a Jacobi-Jordan algebra is a nilalgebra. Since the algebra is commutative, we have
$x^3=x\cdot (x\cdot x)=(x\cdot x)\cdot x$. \\
Natural examples of Jacobi-Jordan algebras are barideals of Bernstein-Jordan algebras, which 
are special instances of train algebras of rank $3$ (see \cite{WOE}, \cite{ZIT} and the references therein.)
However it is not known whether every Jacobi-Jordan algebra arises as a barideal of
some Bernstein-Jordan algebra, see \cite{ZIT}, section $3$. \\
The name of Jordan in the definition of a Jacobi-Jordan algebra is justified by the 
following observation.

\begin{lem}\label{2.2}
A Jacobi-Jordan algebra is a commutative nilalgebra and a Jordan algebra.
\end{lem}

\begin{proof}
Of course we have $x^3=0$ for all $x\in A$ by \eqref{2} if $3\neq 0$. Furthermore by the Jacobi  
identity with $z=x$ we obtain
\begin{align*}
0 & = x^2\cdot y + 2x\cdot (x\cdot y) 
\end{align*}
for all $x,y\in A$. Replacing $y$ by $x\cdot y$ yields

\begin{align*}
x^2\cdot (y\cdot x) & = x^2\cdot (x\cdot y) = -2x\cdot (x\cdot (x\cdot y)) \\
        & = x\cdot (-2x\cdot (x\cdot y)) = x\cdot (x^2\cdot y) \\
        & = (x^2\cdot y)\cdot x 
\end{align*}
for all $x,y\in A$. Hence $A$ is a Jordan algebra. 
\end{proof}

Note that a Jacobi-Jordan algebra is a Jordan algebra in any characteristic. However,
the statement concerning nilalgebra need not be not true in characteristic $3$. For example, 
the $1$-dimensional algebra with basis $e$ and product $e\cdot e=e$ in characteristic $3$ is commutative 
and satisfies the Jacobi identity. 
However, it is not a nilalgebra.

\begin{rem}\label{2.6}
Any Jacobi-Jordan algebra is power-associative, i.e., the subalgebra generated by any element 
is associative. This follows from the fact that a Jacobi-Jordan algebra is a Jordan algebra, see lemma $\ref{2.2}$,
and any Jordan algebra is power-associative, see  \cite{SCH}.
\end{rem}

\begin{lem}
The class of Jacobi-Jordan algebras and commutative nilalgebras of index at most three coincides.
\end{lem}

\begin{proof}
Any Jacobi-Jordan algebra is a commutative nilalgebra of index at most three, because we have $x^3=0$
for all $x\in A$. Conversely, linearizing $x^3=0$ yields $x^2\cdot y+2(x\cdot y)\cdot x=0$ and the
Jacobi identity.
\end{proof}

\begin{lem}
A Jacobi-Jordan algebra is nilpotent.
\end{lem}

\begin{proof}
Albert has shown that any finite-dimensional Jordan nil\-algebra of
characteristic different from $2$ is nilpotent, see \cite{SCH}. This implies that any 
Jacobi-Jordan algebra is nilpotent.
\end{proof}

\begin{rem}
In general, commutative nilalgebras need not be nilpotent. Even commutative power-associative
nilalgebras need not be nilpotent, although this has been conjectured by Albert in $1948$.
Suttles gave the following counterexample of a $5$-dimensional commutative power-associative
nilalgebra with basis $e_1,\ldots ,e_5$ and products
\[
e_1\cdot e_2=e_2\cdot e_4=-e_1\cdot e_5=e_3,\quad  e_1\cdot e_3=e_4, \; e_2\cdot e_3=e_5.
\]
It is solvable, has nilindex $4$, but is not nilpotent, see also the discussion in \cite{GEM}. 
The conjecture was reformulated, in the sense of Albert, namely that every commutative power-associative 
nilalgebra is solvable.
\end{rem}

\section{Classification in dimension at most 4}

There are several classification results of non-associative algebras in low dimensions, which
yield, as a special case, a list of isomorphism classes of Jacobi-Jordan algebras. We assume that
$K$ is an {\it algebraically closed field} of characteristic different from $2$ and $3$, if not said otherwise.
However, in our case of Jacobi-Jordan algebras, it is enough to look at the classification lists 
of commutative, {\it associative} algebras of dimension $n\le 4$. The reason is as follows:

\begin{prop}
A Jacobi-Jordan algebra $A$ of dimension $n\le 4$ is associative, and $A^3=0$.
\end{prop}

\begin{proof}
It is proved in \cite{GEM} that any commutative power-associative nilalgebra $A$ of nilindex
$3$ and of dimension $4$ over a field $K$ of characteristic different from $2$ is associative,
and $A^3=0$. Since by lemma $\ref{2.6}$ any Jacobi-Jordan algebra is power-associative, the claim follows.
It is easy to see that the argument also works in dimension $n\le 3$, for nilindex at most $3$.
\end{proof}

On the other hand, we may also consider the list of nilpotent Jordan algebras in low dimensions,
over an arbitrary field of characteristic different from $2$ and $3$. For algebras of
dimension $n\le 2$ the classification is well-known (it coincides with the classification of
nilpotent Jordan algebras in dimension $2$). In dimension $1$ there is only one 
Jacobi-Jordan algebra, the zero algebra $A_{0,1}$. The second index number
indicates the dimension. In dimension $2$ there are exactly two algebras, the zero algebra 
$A_{0,1}\oplus A_{0,1}$, and $A_{1,2}$, given by the product $e_1\cdot e_1=e_2$. Both algebras are 
Jacobi-Jordan algebras, satisfy $A^3=0$, and hence are associative.
The following result is part of several different classification lists, e.g., of 
\cite{AFM}, \cite{FIP1}, \cite{KAS}, and \cite{MAR}.

\begin{prop}
Any Jacobi-Jordan algebra of dimension $n\le 3$ over an algebraically closed field of characteristic 
different from $2$ and $3$ is isomorphic to exactly one of the following algebras:
\vspace*{0.5cm}
\begin{center}
\begin{tabular}{c|c|c}
Algebra & Products & Dimension \\
\hline
$A_{0,1}$ & $-$ & $1$ \\
\hline
$A_{0,1}\oplus A_{0,1}$ & $-$ & $2$ \\
$A_{1,2}$ & $e_1\cdot e_1=e_2$ & $2$ \\
\hline
$A_{0,1}\oplus A_{0,1}\oplus A_{0,1}$ & $-$ & $3$ \\
$A_{1,2}\oplus A_{0,1}$  & $e_1\cdot e_1=e_2$ & $3$ \\
$A_{1,3}$ & $e_1\cdot e_1=e_2,\; e_3\cdot e_3=e_2$ & $3$ \\
\hline
\end{tabular}
\end{center}
\end{prop}

\begin{rem}
The only indecomposable Jacobi-Jordan algebra in dimension $3$ is $A_{1,3}$. Note that it is
isomorphic over an algebraically closed field of characteristic not $2$ or $3$ to the 
``commutative Heisenberg algebra'', given by the products
\[
e_1\cdot e_2=e_3=e_2\cdot e_1.
\]
In general we have the commutative, associative algebras $H_m$ with basis $(x_1,\ldots ,x_m,y_1,\ldots y_m,z)$
and products $x_i\cdot y_i=y_i\cdot x_i=z$ for $i=1,\ldots ,m$.
\end{rem}

In dimension $4$ we obtain the following classification list, see Table $21$ in \cite{MAR}.

\begin{prop}
Any Jacobi-Jordan algebra of dimension $4$ over an algebraically closed field of characteristic 
different from $2$ and $3$ is isomorphic to exactly one of the following $6$ algebras:
\vspace*{0.5cm}
\begin{center}
\begin{tabular}{c|c}
Algebra & Products \\
\hline
$A_{0,1}^4$ & $-$ \\
$A_{1,2}\oplus A_{0,1}^2$ & $e_1\cdot e_1=e_2$ \\
$A_{1,3}\oplus A_{0,1}$ & $e_1\cdot e_1=e_2,\; e_3\cdot e_3=e_2$ \\
$A_{1,2}\oplus A_{1,2}$ & $e_1\cdot e_1=e_2,\; e_3\cdot e_3=e_4$ \\
\hline
$A_{1,4}$ & $e_1\cdot e_1=e_2,\; e_1\cdot e_3=e_3\cdot e_1=e_4$  \\
$A_{2,4}$ & $e_1\cdot e_1=e_2,\;e_3\cdot e_4=e_4\cdot e_3=e_2$  \\
\hline
\end{tabular}
\end{center}
\end{prop}

There are also exactly six complex Jacobi-Jordan algebras of dimension $4$.
In fact, the classification of all $4$-dimensional complex nilpotent Jordan algebras, 
listed in \cite{AFM}, remains valid for every algebraically closed field of characteristic 
different from $2$ and $3$. This is also noted in \cite{FIP2}, \cite{MAR}.

\begin{rem}
All our Jacobi-Jordan algebras of dimension at most $4$ appear as barideals of
Bernstein-Jacobi algebras, classified by Cortes and Montaner in \cite{CM}.
\end{rem}

\section{Classification in dimension 5}

In dimension $5$ there appears for the first time a Jacobi-Jordan algebra which is not
associative. The classification of nilpotent, commutative associative algebras of dimension $n=5$
over an algebraically closed field can be found in Poonen's article \cite{POO}. For
characteristic $\neq 2, 3$ there are $25$ algebras, corresponding to
the list of local algebras in dimension $6$, see table $1$ in \cite{POO}. If we ``forget'' 
the unit element, we obtain the list of $5$-dimensional algebras without unit.
The classification list of $5$-dimensional associative Jacobi-Jordan algebras is a sublist of these
$25$ algebras, consisting of $14$ algebras. We first show that there is only one 
Jacobi-Jordan algebra in dimension $5$ which is {\it not} associative.

\begin{prop}
Any Jacobi-Jordan algebra of dimension $5$ over an algebraically closed field of characteristic
different from $2$ and $3$ is associative or isomorphic to the following algebra, given by
the products of the basis $\{e_1,\ldots ,e_5 \}$
as follows:
\begin{align*}
e_1\cdot e_1 & = e_2,\; e_1\cdot e_4=e_4\cdot e_1=e_5,\\
e_1\cdot e_5 & = e_5\cdot e_1 =-\frac{1}{2}e_3,\; e_2\cdot e_4=e_4\cdot e_2=e_3.
\end{align*}
This algebra is not associative. It is, up to isomorphism, the unique nilpotent Jordan algebra 
$A$ of nilindex $3$ with $A^3\neq 0$.
\end{prop}

\begin{proof}
Let $A$ be a $5$-dimensional Jacobi-Jordan algebra. If $A^3=0$, then $A$ is associative.
Now assume that $A^3\neq 0$. Then $A$ has nilindex $3$, because otherwise $A^2=0$. Hence we can 
use the arguments given in \cite{ELS} for the case of Jordan 
nilalgebras of nilindex $3$. There is only one such algebra up to isomorphism.
There exists $x,y\in A$ such that $y\cdot x^2\neq 0$. Using the identities 
\begin{align*}
x^2\cdot y+2(x\cdot y)\cdot x & = 0, \\
x^2\cdot (x\cdot y) & = 0,\\
x^2\cdot (y\cdot z)+2(x\cdot y)\cdot(x\cdot z) & = 0,\\
2(x\cdot y)^2+x^2\cdot y^2 & = 0,
\end{align*}
we see that $\{y,x,y\cdot x^2,x^2,y\cdot x\}$ is a basis of $A$. Since $\dim (A^2)\le 3$ by Lemma $1.4$
of \cite{ELS}, we have $A^2={\rm span}\{y\cdot x^2,x^2,y\cdot x \}$. Since $y^2\in A^2$ there exist
$\al,\be,\ga \in K$ such that 
\[
y^2=\al y\cdot x^2+\be x^2+\gamma y\cdot x.
\]
A direct computation using the above identities shows that $4\be=-\ga^2$, see \cite{ELS}. By setting
\[
e_1=x,\; e_2=x^2, \; e_3=y\cdot x^2,\; e_4=y-\frac{1}{2}\al x^2-\frac{1}{2}\ga x, \; 
e_5=y\cdot x-\frac{1}{2}\ga x^2
\]
we see that $\{e_1,\ldots ,e_5\}$ is a basis of $A$ with the above products. This algebra 
is a Jordan-Jacobi algebra which is not associative, because $(e_1\cdot e_1)\cdot e_4=e_3$ and
$e_1\cdot (e_1\cdot e_4)=-\frac{1}{2}e_3$.
\end{proof}

We will explain now by example how to read Poonen's classification list \cite{POO}, enabling us to
obtain a list of $5$-dimensional associative Jacobi-Jordan algebras. Isomorphism types of such algebras
are in bijection with isomorphism types of local commutative associative algebras with $1$, of rank $6$ and of
nilindex $\le 3$.

\begin{ex}
The local unital commutative associative algebra of rank $6$ in Table $1$ of \cite{POO}, given by
$K[x,y]/I$ with ideal $I=(xy,yz,z^2,y^2-xz,x^3)$ yields the $5$-dimensional associative Jacobi-Jordan
algebra given by the products
\[
e_1\cdot e_1=e_2,\; e_1\cdot e_4=e_4\cdot e_1=e_5,\; e_3\cdot e_3=e_5.
\]
\end{ex}

To see this we compute a Groebner basis for the ideal $I$ in $K[x,y]$, with respect to the
standard lexicographical order: $\CG=\{x^3,xy,xz-y^2,y^3,yz,z^2 \}$. Therefore we obtain a basis
for the quotient $K[x,y]/I$ by $\CB=\{x,x^2,y,z,y^2 \}=\{e_1,\ldots ,e_5 \}$. We use here, that the 
monomials which are not divisible by a leading monomial from $\CG$ form a basis of the quotient.
Now we compute the products $e_i\cdot e_j$. We have $e_1\cdot e_1=x^2=e_2$, $e_1\cdot e_2=x^3=0$,
$e_1\cdot e_3=xy=0$, $e_1\cdot e_4=xz=y^2=e_3$, $e_1\cdot e_5=xy^2=0$, and so forth. 

\begin{prop}
Any indecomposable associative Jacobi-Jordan algebra of dimension $5$ over an algebraically closed 
field of characteristic not $2$ and $3$ is isomorphic to exactly one of the following $7$ algebras:
\vspace*{0.5cm}
\begin{center}
\begin{tabular}{c|c|c}
Algebra & Products & Ideal $I$ in \cite{POO}\\
\hline
$A_{1,5}$ & $e_1\cdot e_1=e_2,\; e_1\cdot e_3=e_3\cdot e_1=e_5,\; e_3\cdot e_3=e_4$ & $(x,y)^3$ \\
$A_{2,5}$ & $e_1\cdot e_1=e_2,\; e_1\cdot e_4=e_4\cdot e_1=e_5,\; e_3\cdot e_3=e_5$ & $(xy,yz,z^2,y^2-xz,x^3)$\\
$A_{3,5}$ & $e_1\cdot e_1=e_2,\; e_1\cdot e_4=e_4\cdot e_1=e_5,$ & $(xy,z^2,xz-yz,x^2+y^2-xz)$ \\
         & $e_3\cdot e_3=-e_2+e_5, \; e_3\cdot e_4=e_4\cdot e_3=e_5 $ & \\
$A_{4,5}$ & $e_1\cdot e_1=e_2,\; e_3\cdot e_3=e_4,\; e_5\cdot e_5=-e_2+e_4$ & $(xy,xz,yz,x^2+y^2-z^2)$ \\
$A_{5,5}$ & $e_1\cdot e_1=e_2,\; e_3\cdot e_3=e_4,\; e_3\cdot e_5=e_5\cdot e_3=-e_2+e_4$ & $(x^2,xy,yz,xz+y^2-z^2)$ \\
$A_{6,5}$ & $e_1\cdot e_1=e_2,\;e_1\cdot e_3=e_3\cdot e_1=e_5,$  & $(x^2,xy,y^2,z^2)$\\
         & $e_1\cdot e_4=e_4\cdot e_1 =\frac{1}{2}e_2$ & \\
$A_{7,5}$ & $e_1\cdot e_1=e_2,\;e_3\cdot e_3=-e_2,$  & $(x^2,y^2,z^2,w^2,xy-zw,xz,$ \\
         & $e_4\cdot e_5=e_5\cdot e_4=\frac{1}{2}e_2$ & $xw,yz,yw)$\\
\hline
\end{tabular}
\end{center}
\end{prop}
\vspace*{0.5cm}
For decomposable algebras we have also identified the associative Jacobi-Jordan algebras among
Poonen's list in table $1$ of \cite{POO}. The result is as follows.

\begin{prop}
Any decomposable associative Jacobi-Jordan algebra of dimension $5$ over an algebraically closed 
field of characteristic not $2$ and $3$ is isomorphic to exactly one of the following $7$ algebras:
\vspace*{0.5cm}
\begin{center}
\begin{tabular}{c|c|c}
Algebra & Products & Ideal $I$ in \cite{POO}\\
\hline
$A_{0,1}^5$ & $-$ & $(x,y,z,w,v)^2$ \\
$A_{1,2}\oplus A_{0,1}^3$ & $e_1\cdot e_1=e_2$ & $(x^2,y^2,z^2,xy,xz,xw,yz,yw,zw,w^3)$\\
$A_{1,3}\oplus A_{0,1}^2$ & $e_1\cdot e_1=e_2,\; e_3\cdot e_3=e_2$ & $(x^2,y^2,z^2,w^2,xy,xz,xw,yz,yw)$ \\
$A_{2,4}\oplus A_{0,1}$ & $e_1\cdot e_1=e_2,\; e_3\cdot e_4=e_4\cdot e_3=e_2$ & $(x^2,y^2+zw,z^2,w^2,xy,xz,xw,yz,yw)$ \\
$A_{1,2}^2\oplus A_{0,1}$ & $e_1\cdot e_1=e_2,\; e_3\cdot e_3=e_4$ & $(x^2,xy,xz,yz,y^3,z^3)$ \\
$A_{1,2}\oplus A_{1,3}$ & $e_1\cdot e_1=e_2,\;e_3\cdot e_3=e_4,\; e_5\cdot e_5=e_4$  & $(xy,xz,y^2,z^2,x^3)$\\
$A_{1,4}\oplus A_{0,1}$ & $e_1\cdot e_1=e_2,\;e_1\cdot e_3=e_3\cdot e_1=e_4$  & $(x^2,xy,xz,y^2,yz^2,z^3)$ \\
\hline
\end{tabular}
\end{center}
\end{prop}
\vspace*{0.5cm}

\section{Classification in higher dimensions}

In dimension $6$ there are already infinitely many isomorphism classes of 
commutative, associative algebras over any algebraically closed field, see \cite{POO}. 
Therefore it is natural to stop here with a classification of Jacobi-Jordan
algebras. On the other hand some subclasses still have a nice classification. In particular we may
consider the classification of Jacobi-Jordan algebras which are {\it not} associative. \\
In dimension $6$ we still have only finitely many different "non-associative" Jacobi-Jordan algebras.

\begin{prop}
Any non-associative Jacobi-Jordan algebra of dimension $6$  over an algebraically closed 
field of characteristic not $2$ and $3$ is isomorphic to exactly one of the following $5$ algebras,
with $\be,\de\in \{0,1 \}$:
\vspace*{0.5cm}
\begin{center}
\begin{tabular}{c|c|c}
Algebra $A$ & Products & $\dim A^2$ \\
\hline
$A_{1,6}$ & $e_1\cdot e_1=e_2,\; e_1\cdot e_3=e_3\cdot e_1=e_5,\; e_1\cdot e_4=e_4\cdot e_1=e_6,$ & $4$\\
         & $e_3\cdot e_3=e_4,\; e_3\cdot e_5=e_5\cdot e_3=-\frac{1}{2}e_6$ & \\
\hline
$A_{2,6}(\be,\de)$ & $e_1\cdot e_1=e_2,\; e_1\cdot e_3=e_3\cdot e_1=-\be e_2,\;e_1\cdot e_4=e_4\cdot e_1=e_5,$ & $3$ \\
$\be,\de \in\{0,1\}$  & $e_1\cdot e_5=e_5\cdot e_1=-\frac{1}{2}e_6,\; e_2\cdot e_4=e_4\cdot e_2=e_6$ & \\
                  & $e_3\cdot e_3=\de e_6,\; e_3\cdot e_5=e_5\cdot e_3=\be e_6$ & \\                
\hline
\end{tabular}
\end{center}
\end{prop}
\vspace*{0.5cm}
\begin{proof}
It is shown in \cite{ELS2} that a Jordan nilalgebra of nilindex $3$ satisfies $A^3=0$, and hence
is associative, or satisfies $A^3\neq 0$ and $3\le \dim A^2\le 4$. For $\dim A^2=4$ and $A^3\neq 0$ there is 
only one algebra \cite{ELS2}, which is isomorphic to $A_{1,6}$. It is indeed not associative, 
since we have $(e_1\cdot e_3)\cdot e_3=-\frac{1}{2}e_6$ and $e_1\cdot (e_3\cdot e_3)=e_6$. \\
For $\dim A^2=3$ there is a $2$-parameter family of algebras $A(\be,\de)$. It is shown in \cite{ELS2}
that these are the only ones with $\dim A^2=3$, but the isomorphism 
classes of the algebras $A(\be,\de)$ are not determined. However, it is not difficult to see by 
explicit calculations that we obtain exactly 
four algebras, i.e., $A(0,0),\, A(0,1),\, A(1,0),\, A(1,1)$ up to isomorphism. Let us give an example.
For $\be\neq 0$ we have an algebra isomorphism $\phi\colon A(\be,0)\simeq A(1,0)$ given by
\[
\phi=\begin{pmatrix}
2\be        & 0    & 0  & 0                & 0  & 0 \\
0           & 0    & 0  & 0                & -1 & 0 \\
1           & 0    & 0  & \frac{1}{2\be^2} & 0  & 0 \\
\frac{1}{2} & 0    & -1 & 0                & 0  & 0 \\
0           & 2\be & 0  & 0                & 0  & 0 \\
0           & 0    & 0  & 0                & 0  & 1 \\
\end{pmatrix},
\]
with $\det(\phi)=2$. Similarly we obtain $A(0,\de)\simeq A(0,1)$ for $\de\neq 0$ and  $A(\be,\de)\simeq A(1,1)$
for $\be,\de\ne 0$. We denote these algebras by $A_{2,6}(\be,\de)$ in the above table. They are not associative,
since $(e_4\cdot e_1)\cdot e_1=-\frac{1}{2}e_6$ and $e_4\cdot (e_1\cdot e_1)=e_6$. 
\end{proof}

\end{document}